\documentclass[11pt]{article}
\usepackage{mathtools}
\usepackage{amssymb}
\usepackage{amsthm}
\usepackage{tikz-cd}
\usepackage{cite}
\usepackage[margin=1.5in]{geometry}
\usepackage{parskip}
\usepackage{mathrsfs}
\usepackage{hyperref}
\usepackage{float}

\DeclareSymbolFont{largesymbols}{OMX}{yhex}{m}{n}
\DeclareMathAccent{\wideparen}{\mathord}{largesymbols}{"F3}

\theoremstyle{plain}
\newtheorem{thm}{Theorem}[subsection]
\theoremstyle{definition}
\newtheorem{defn}[thm]{Definition}

\newtheorem{lemma}[thm]{Lemma}
\newtheorem{prop}[thm]{Proposition}
\newtheorem{cor}[thm]{Corollary}
\newtheorem{conj}[thm]{Conjecture}

\newtheorem{MainThm}{Theorem}

\title{Normality of the dual nilcone in positive characteristic}
\author{Richard Mathers}

\begin{document}

\setcounter{secnumdepth}{3}
\setcounter{tocdepth}{3}

\maketitle

\begin{abstract}
  
Let $G$ be a connected semisimple algebraic group of adjoint type defined over an algebraically closed field $K$ of positive characteristic. We demonstrate that the dual nilcone $\mathcal{N}^* \subseteq \mathfrak{g}^*$ is a normal variety in certain positive characteristics which are not very good, and that in these cases the Springer map $\mu: T^*\mathcal{B} \to \mathcal{N}$ is a resolution of singularities. \\

As an application, we extend the results of Ardakov and Wadsley on representations of $p$-adic Lie groups. Under further restrictions on the characteristic, we show that the canonical dimension of a coadmissible representation of a semisimple $p$-adic Lie group in a $p$-adic Banach space is either zero or at least half the dimension of a nonzero coadjoint orbit.
\end{abstract}

\tableofcontents

\newpage

\section{Introduction}

Let $G$ be a connected semisimple algebraic group of adjoint type defined over an algebraically closed field $K$ of positive characteristic $p > 0$. In case the characteristic of $K$ is very good for $G$, which, broadly speaking, implies that $G$ is not of type $A$ and $p > 5$, it is known that the dual nilpotent cone $\mathcal{N}^* \subseteq \mathfrak{g}^*$ is a normal variety, and it admits a desingularisation $\mu: T^*\mathcal{B} \to \mathcal{N}$ from the cotangent bundle of the flag variety $\mathcal{B}$ of $G$; the so-called $\emph{Springer resolution}$ of $\mathcal{N}$. \\

When $p$ is small, the classical proofs of these results break down. The goal of this paper is to investigate in which bad characteristics the dual nilcone $\mathcal{N}^*$ remains a normal variety and the Springer map is a resolution of singularities. \\

In case $G$ is of type $A_n$, the picture is a little different. Here, the classical proofs are valid when $p$ does not divide $n+1$. We have the following main theorems: \\ 

\begin{MainThm}\label{thm a}
Let $G = PGL_n$ and suppose $p|n$. Then the dual nilpotent cone $\mathcal{N}^* \subseteq \mathfrak{g}^*$ is a normal variety.
\end{MainThm} \vspace{4.4 mm}

\begin{MainThm}\label{thm b}
Suppose $G$ is of type $G_2$ and $p = 2$. Then the dual nilpotent cone $\mathcal{N}^* \subseteq \mathfrak{g}^*$ is a normal variety.
\end{MainThm}

As an application, let $p$ be a prime,  $G$ be a semisimple compact $p$-adic Lie group and let $K$ be a finite extension of $\mathbb{Q}_p$. Ardakov and Wadsley studied the coadmissible representations of $G$, which are finitely generated modules over the completed group ring $KG$ with coefficients in $K$, in \cite{AW}. These completed group rings may be realised as Iwasawa algebras, which are important objects in noncommutative Iwasawa theory. \\

One of the central results in \cite{AW} is an estimate for the canonical dimension of a coadmissible representation of a semisimple $p$-adic Lie group in a $p$-adic Banach space. When $p$ is very good for $G$, Ardakov and Wadsley showed that this canonical dimension is either zero or at least half the dimension of a nonzero coadjoint orbit. We extend their results to the case where $G = PGL_n$, $p|n$, and $n > 2$. The main result of this section is as follows: \\

\begin{MainThm}\label{thm c}
Let $G$ be a compact $p$-adic analytic group whose Lie algebra is semisimple. Suppose that $G = PGL_n$, $p|n$, and $n > 2$. Let $G_{\mathbb{C}}$ be a complex semisimple algebraic group with the same root system as $G$, and let $r$ be half the smallest possible dimension of a nonzero coadjoint $G_{\mathbb{C}}$-orbit. Then any coadmissible $KG$-module $M$ that is infinite-dimensional over $K$ satisfies $d(M) \geq r$.
\end{MainThm}

$\textbf{Acknowledgments.}$ I would like to thank Konstantin Ardakov for suggesting this research project, and Kevin McGerty for his interest in my work and his helpful contributions.

\section{The nilpotent cone and the Springer resolution}\label{nilcone chapter}

\subsection{Characteristic}

In this section, we study the geometric structure of the nilpotent cone $\mathcal{N}$ of the Lie algebra $\mathfrak{g}$ of a reductive algebraic group $G$ in arbitrary characteristic. We begin with a discussion of the ordinary nilpotent cone, defined as a subvariety of $\mathfrak{g}$, and then give a characterisation of the dual nilpotent cone $\mathcal{N}^*$. 

Our treatment of the material on $\mathcal{N}$ is based on that of Jantzen in \cite{JN}. We generalise some of his arguments which are dependent on certain restrictions on the characteristic. Later, we will specialise further to the case $G = PGL_n$ and $p|n$ at certain points of the argument. The last subsection of the section discusses analogues of the results presented here when we consider a more general algebraic group $G$. \\

Let $\textbf{G}$ be a split reductive algebraic group scheme, defined over $\mathbb{Z}$, and $K$ an algebraically closed field of characteristic $p > 0$. Let $G := \textbf{G}(K)$. Let $\mathfrak{g}$ denote the Lie algebra of $G$ and $W(G)$ the Weyl group of $G$. When $G$ is clear from context, we will abbreviate $W(G)$ to $W$. Since $G$ is a linear algebraic group, we fix an embedding $G \subseteq GL(V)$ for some $n$-dimensional $K$-vector space $V$. \\

\begin{defn}
Let $\alpha_i$ be the simple roots of the root system $R$ of $G$, and let $\beta$ be the highest-weight root. Writing $\beta = \sum_i m_i \alpha_i$, $p$ is $\textit{bad}$ for $G$ if $p = m_i$ for some $i$. $p$ is $\emph{good}$ if $p$ is not bad. \\

The prime $p$ is $\emph{very good}$ if one of the following conditions hold: \\

(a) $G$ is not of type $A$ and $p$ is good, \\
(b) $G$ is of type $A_n$ and $p$ does not divide $n+1$. \\
\end{defn}

In practice, we have the following classification. In types $B, C$ and $D$, the only bad prime is 2. For the exceptional Lie algebras, the bad primes are 2 and 3 for types $E_6, E_7, F_4$ and $G_2$, and 2,3 and 5 for type $E_8$. In type $A$, there are no bad primes. For more details of this classification, see \cite[I.4.3]{SS}. \\

\begin{defn}\label{nonspecial defn}
A prime $p$ is $\emph{special}$ for $G$ if the pair (Dynkin diagram of $G$, $p$) lies in the following list:

(a) ($B$, 2), \\
(b) ($C$, 2), \\
(c) ($F_4$, 2), \\
(d) ($G_2$, 3).

A prime $p$ is $\emph{nonspecial}$ for $G$ if it is not special.
\end{defn}

This definition, and material on the importance of nonspecial primes, can be found in \cite[Section 5.6]{PS}.

\subsection{The $W$-invariants of $S(\mathfrak{h})$}

Let $G = PGL_n$ and suppose $p|n$. This short section investigates the structure of the invariants of the Weyl group action on the symmetric algebra $S(\mathfrak{h})$. \\

Let $\mathfrak{g}^*$ be the dual vector space of $\mathfrak{g}$. Since $G$ is of type $A$ and the prime $p$ is always good for $G$, there is a $G$-equivariant isomorphism $\kappa: \mathfrak{g} \to \mathfrak{g}^*$ by the argument in \cite[Section 6.5]{JN}. Since $\mathfrak{g}$ is a finite-dimensional vector space, we naturally identify the symmetric algebra $S(\mathfrak{g})$ and the algebra of polynomial functions $K[\mathfrak{g}^*]$. \\
 
Let $\mathfrak{h}$ be a fixed Cartan subalgebra of $\mathfrak{g}$. The Weyl group $W$ has a natural action on $\mathfrak{h}$, which can be extended linearly to an action of $W$ on the symmetric algebra $S(\mathfrak{h})$. The identification $S(\mathfrak{h}) \cong K[\mathfrak{h}^*]$ is compatible with the $W$-action. We begin this section by studying the $W$-invariants under this action. \\

\begin{thm}\label{polynomial theorem}
Suppose $G = PGL_n$ and $p|n$. Then $S(\mathfrak{h})^W$ is a polynomial ring.
\end{thm}

\begin{proof}
Recall the Weyl group $W$ is isomorphic to $S_n$, and let $\mathfrak{t}$ be the image of the diagonal matrices in $\mathfrak{g} = \mathfrak{pgl}_n$. Then $\mathfrak{t}$ is the quotient of the natural $S_n$-module $V$ with basis $\lbrace e_1, \cdots, e_n \rbrace$, permuted by $S_n$, by the trivial submodule $U := K(\sum_{i=1}^n e_i)$. Let $X = V/U$. The quotient map $V \to X$ induces a surjective map $S(V) \to S(X)$. \\

Suppose $p = n = 2$ and let $\lbrace \overline{e_1}, \overline{e_2} \rbrace$ be the images of the vector space basis $\lbrace e_1, e_2 \rbrace$ of $V$ inside $X$. Let $\sigma$ denote the non-identity element of $S_2$. Then $\sigma \cdot \overline{e_1} = \overline{e_2}$ and $\sigma \cdot \overline{e_2} = \overline{e_1}$. Since $\overline{e_1} + \overline{e_2} = 0$, it follows that $\overline{e_1} = \overline{e_2}$. Hence $S(X)^{S_2} = S(X)$, which is a polynomial ring. \\

Now suppose $n > 2$ and $p|n$. We claim that the $S_n$-action on $V$ and on $X$ is faithful. The $S_n$-action on $V$ is by permutation and therefore is faithful. To see the claim for the $S_n$-action on $X$, let $N := \lbrace g \in S_n \mid g \cdot x = x \text{ } \forall x \in S(X) \rbrace$ denote the kernel of the natural map $S_n \to S(X)$. \\

Suppose $g$ is some non-identity element of $N$. Then, relabelling the elements $\overline{e_i}$ if necessary, $g \cdot \overline{e_1} = \overline{e_2}$. Hence it suffices to show that $\overline{e_1} \neq \overline{e_2}$. If $\overline{e_1} = \overline{e_2}$, then since $\sum_{i=1}^n \overline{e_i} = 0$, $\sum_{i=2}^n \overline{e_i} = \overline{e_1}$ and $e_1 + \sum_{i=3}^n \overline{e_i} = \overline{e_2}$. Rearranging, $\sum_{i=3}^n \overline{e_i} = (p-2)\overline{e_1}$. Hence the set $\lbrace \overline{e_1}, \overline{e_3}, \cdots, \overline{e_{n-1}} \rbrace$ spans $X$, but $X$ is an $(n-1)$-dimensional vector space, a contradiction. It follows that the $S_n$-action on $X$ is faithful. \\

The ring of invariants $S(V)^{S_n}$ is generated by the elementary symmetric polynomials $s_1(e_1, \cdots,  e_n), \cdots, s_n(e_1, \cdots, e_n)$, which are algebraically independent by \cite[Section 6, Theorem 1]{Bo3}. Applying \cite[Proposition 4.1]{Na}, we see that $S(X)^{S_n}$ is also a polynomial ring. The proof of \cite[Proposition 5.1]{KM} also demonstrates that $S(X)^{S_n}$ is generated by the images of $s_2(e_1, \cdots,  e_n), \cdots, s_n(e_1, \cdots, e_n)$ under the map $S(V) \to S(X)$. \\

To finish the proof, it suffices to note that we may identify $\mathfrak{t} \cong \mathfrak{h}$ and that $\mathfrak{h} \cong V/U = X$.
\end{proof}

We state a version of Kostant's freeness theorem that will be useful for our applications. \\

\begin{thm}\label{freeness polynomial}
$S(\mathfrak{h})$ is a free $S(\mathfrak{h})^W$-module if and only if $S(\mathfrak{h})^W$ is a polynomial ring.
\end{thm}

\begin{proof}
See \cite[Corollary 6.7.13]{Sm}.
\end{proof}

\subsection{Properties of the nilpotent cone}

We now outline some general preliminaries on the structure theory of groups acting on varieties. At first, we do not impose any restriction on the characteristic. \\

Let $M$ be a variety which admits an algebraic group action by $G$, and let $x \in M$. The closure $\overline{Gx}$ of the orbit $Gx$ of $x$ is a closed subvariety of $M$. By \cite[Proposition 8.3]{Hu2}, $Gx$ is open in $\overline{Gx}$ and so $Gx$ has the structure of an algebraic variety. \\

The orbit map $\pi_x: G \to Gx$, $\pi_x(g) = gx$, is a surjective morphism of varieties. The stabiliser $G_x := \lbrace g \in G \mid gx = x \rbrace$ is a closed subgroup of $G$, and $\pi_x$ induces a bijective morphism:

\begin{gather*}
\overline{\pi_x}: G/G_x \to Gx
\end{gather*}

by \cite[Section 12]{Hu2}. \\

We now specialise to the case where $M = \mathfrak{g}$ and $G$ acts on $\mathfrak{g}$ via the adjoint action. Let $X \in \mathfrak{g}$ and let $GX$ denote the $G$-orbit of $X$ under the adjoint action $\text{Ad}: G \to \text{Ad}(\mathfrak{g})$. \\

Recall that an element $g \in \mathfrak{g}$ is $\emph{nilpotent}$ if the operator $\text{ad}_g(y)$ is nilpotent for each $y \in \mathfrak{g}$. The set of nilpotent elements is denoted $\mathcal{N}$. 

Since $G$ is a linear algebraic group, fix an embedding $G \subseteq GL(V)$ for some $n$-dimensional $K$-vector space $V$. Then $\mathcal{N} = \mathfrak{g} \cap \mathcal{N}(\mathfrak{gl}(V))$, where $\mathcal{N}(\mathfrak{gl}(V))$ denotes the set of nilpotent elements of the Lie algebra of $GL(V)$. It follows that $\mathcal{N}$ is closed in $\mathfrak{g}$, and hence $\mathcal{N}$ has the structure of a subvariety of the algebraic variety $\mathfrak{g}$. \\

Let:

\begin{gather*}
P_X(t) := \text{det}(tI - X)
\end{gather*}

denote the characteristic polynomial of $X$ in the variable $t$. Then:

\begin{gather*}
P_X(t) := t^n + \sum_{i=1}^n (-1)^i s_i(X) t^{n-i}
\end{gather*}

where each $s_i$ is a homogeneous polynomial of degree $i$ in the entries of $X$. If $a_1, \cdots, a_n$ are the eigenvalues of $X$, counted with algebraic multiplicity, then, since $K$ is algebraically closed, $P_X(t) = \prod_{i=1}^n (t-a_i)$, and so $s_i(X)$ can be identified with the $i$th elementary symmetric function in the $a_j$. It follows that $X$ is nilpotent if and only if $P_X(t) = t^n$ if and only if $s_i(X) = 0$ for each $i$:

\begin{gather*}
\mathcal{N}(\mathfrak{gl}(V)) = \lbrace X \in \mathfrak{gl}(V) \mid s_i(X) = 0 \text{ for all } i \rbrace.
\end{gather*}

Let $S(V)$ denote the algebra of polynomial functions on $V$. This has a natural grading by degree, with $S(V) = \bigoplus_{i \geq 0} S^i(V)$. Set $S^+(V) := \bigoplus_{i \geq 1} S^i(V)$. \\

Now the restrictions of the $s_i$ to $\mathfrak{g}$ are $G$-invariant polynomial functions on $\mathfrak{g}$, and so $s_{i|\mathfrak{g}} \in S^i(\mathfrak{g}^*)^G$. It follows that there exist $f_1, \cdots, f_n \in S^+(\mathfrak{g}^*)^G$ such that:

\begin{gather*}
\mathcal{N} = \lbrace X \in \mathfrak{g} \mid f_i(X) = 0 \text{ for all } i \rbrace.
\end{gather*}
 
\begin{prop}\label{affine nilpotent cone}
The nilpotent cone $\mathcal{N}$ may be realised as:

\begin{gather*}
\mathcal{N} = \lbrace X \in \mathfrak{g} \mid f(X) = 0 \text{ for all } f \in S^+(\mathfrak{g}^*)^G \rbrace.
\end{gather*}

Hence $\mathcal{N} = V(S^+(\mathfrak{g}^*)^G)$ is an affine variety.
\end{prop} 

\begin{proof}
It is clear that $\lbrace X \in \mathfrak{g} \mid f(X) = 0 \text{ for all } f \in S^+(\mathfrak{g}^*)^G \rbrace \subseteq \mathcal{N}$ by the above discussion. Conversely, given $f \in S^+(\mathfrak{g}^*)^G$, $f(0) = 0$ and $f$ is constant on the closure of the orbits under the adjoint action. Then $f$ is constant on $\overline{GX}$, the closure of the regular orbit under the adjoint action, and $0 \in \overline{GX}$ by \cite[Proposition 2.11(1)]{JN}.
\end{proof}

\begin{lemma}\label{borel subgp variety}
Let $\mathcal{B}$ be the set of all Borel subalgebras of $\mathfrak{g}$. Then there is a bijection $G/B \leftrightarrow \mathcal{B}$.
\end{lemma}

\begin{proof}
$\mathcal{B}$ is the closed subvariety of the Grassmannian of $\text{dim } \mathfrak{b}$-dimensional subspaces in $\mathfrak{g}$ formed by solvable Lie algebras. Hence $\mathcal{B}$ is a projective variety. All Borel subalgebras are conjugate under the adjoint action of $G$, and the stabiliser subgroup $G_{\mathfrak{b}}$ of $\mathfrak{b}$ in $G$ is equal to $B$ by \cite[Theorem 11.16]{Bo}. Hence the claimed bijection follows via the assignment $g \mapsto g \cdot \mathfrak{b} \cdot g^{-1}$.
\end{proof}

\begin{defn}\label{enhanced nilpotent cone}
Set $\widetilde{\mathfrak{g}} := \lbrace (x, \mathfrak{b}) \in \mathfrak{g} \times \mathcal{B} \mid x \in \mathfrak{b} \rbrace$, and let $\mu: \widetilde{\mathfrak{g}} \to \mathfrak{g}$ be the projection onto the first coordinate. The $\emph{enhanced nilpotent cone}$ is the preimage of $\mathcal{N}$ under the map $\mu$:

\begin{gather*}
\widetilde{\mathcal{N}} := \mu^{-1}(\mathcal{N}) = \lbrace (x, \mathfrak{b}) \in \mathcal{N} \times \mathcal{B} \mid x \in \mathfrak{b} \rbrace.
\end{gather*}
\end{defn}

\begin{lemma}\label{enhanced nilpotent cone smooth}
$\widetilde{\mathcal{N}}$ is a smooth irreducible variety.
\end{lemma}

\begin{proof}
Let $\mathfrak{b} \in \mathcal{B}$ be a fixed Borel subalgebra. The fibre over $\mathfrak{b}$ of the second projection $\pi: \widetilde{\mathcal{N}} \to \mathcal{B}$ is the set of nilpotent elements of $\mathfrak{b}$. Decomposing $\mathfrak{b} = \mathfrak{h} \oplus \mathfrak{n}$, where $\mathfrak{n}:= [\mathfrak{b}, \mathfrak{b}]$ is the nilradical of $\mathfrak{b}$, an element $x \in \mathfrak{b}$ is nilpotent if and only if it is has no component in the Cartan subalgebra $\mathfrak{h}$. Hence $\pi$ makes $\widetilde{\mathcal{N}}$ a vector bundle over $\mathcal{B}$ with fibre $\mathfrak{n}$. \\

The canonical map $G \to G/B$ is locally trivial by \cite[II.1.10(2)]{Ja2}, so the set of $B$-orbits on $G \times \mathfrak{n}$ has a natural structure of a variety, denoted $G \times_B \mathfrak{n}$. The above construction yields a $G$-equivariant vector bundle isomorphism:

\begin{gather*}
\widetilde{\mathcal{N}} \cong G \times_B \mathfrak{n},
\end{gather*}

where $B$ is the Borel subgroup of $G$ corresponding to $\mathfrak{b}$. It follows that we may view $\widetilde{\mathcal{N}}$ as a vector bundle over the smooth variety $G/B$, and so $\widetilde{\mathcal{N}}$ is smooth. \\

Using Lemma \ref{borel subgp variety}, identify $\mathcal{B}$ with $G/B$ and consider the morphism $f: \mathfrak{g} \times G \to \mathfrak{g} \times \mathcal{B}$ defined by $f(x,g) = (x,gB)$. The inverse image:

\begin{gather*}
f^{-1}(\widetilde{\mathcal{N}}) = \lbrace (x,g) \in \mathcal{N} \times G \mid \text{Ad}(g^{-1})(x) \in \mathfrak{n} \rbrace
\end{gather*}

is closed in $\mathfrak{g} \times G$ since it is the inverse image of $\mathfrak{n}$ under the natural map $\mathfrak{g} \times G \to \mathfrak{g}$, $(x,g) \mapsto \text{Ad}(g^{-1})(x)$. Since $f$ is an open map and $f^{-1}(\widetilde{\mathcal{N}})$ is closed, $\widetilde{\mathcal{N}}$ is a closed subvariety of $\mathcal{N} \times \mathcal{B}$. \\

The morphism $\mathfrak{n} \times G \to \widetilde{\mathcal{N}}, (x,g) \mapsto (\text{ad}(g)(x), gB)$ is surjective by definition. Hence $\widetilde{\mathcal{N}}$ is irreducible.
\end{proof}

By \cite[Theorem 2.8(1)]{JN}, there are only finitely many orbits for the $G$-action in the nilpotent cone $\mathcal{N}$. Let $X_1, \cdots, X_r$ be representatives for these orbits. Then:

\begin{gather*}
\mathcal{N} = \bigcup_{i=1}^r \overline{\mathcal{O}_{X_i}}
\end{gather*}

Since $\mathcal{N}$ is irreducible by Lemma \ref{enhanced nilpotent cone smooth}, one of these closed subsets must be all of $\mathcal{N}$: let $\overline{GZ} = \mathcal{N}$. Then, by \cite[1.13, Corollary 1]{St3}, this orbit is open in $\mathcal{N}$ and $\text{dim}(GZ) = \text{dim}(\mathcal{N})$, while $\text{dim}(GY) < \text{dim}(GZ)$ for any $GY \neq GZ$. Hence $GZ$ is unique with respect to this property. \\

\begin{defn}\label{regular element}
An element $X \in \mathfrak{g}$ is $\emph{regular}$ if it lies in $GZ$, the unique open dense $G$-orbit of $\mathcal{N}$.
\end{defn} 

We now specialise to the case where $G = PGL_n$ and $p|n$. \\

\begin{lemma}\label{enhanced nilpotent cone cotangent}
There is a natural $G$-equivariant vector bundle isomorphism:

\begin{gather*}
\widetilde{\mathcal{N}} \cong T^*\mathcal{B}.
\end{gather*}
\end{lemma}

\begin{proof}
This follows from \cite[Section 6.5]{JN}.
\end{proof}

\begin{defn}
The map $\mu: T^*\mathcal{B} \to \mathcal{N}$ is the $\emph{Springer resolution}$ for the nilpotent cone $\mathcal{N}$.
\end{defn} \vspace{4.4 mm}

\begin{lemma}\label{density}
Let $\mathcal{N}_s$ denote the set of smooth points of $\mathcal{N}$. Then $\mu^{-1}(\mathcal{N}_s)$ is dense in $\widetilde{\mathcal{N}}$. 
\end{lemma}

\begin{proof}
$\mathcal{N}_s$ is an open and non-empty subset of $\mathcal{N}$. Hence it is dense, and its preimage is open and non-empty in $\widetilde{\mathcal{N}}$. By Lemma \ref{enhanced nilpotent cone smooth}, $\widetilde{\mathcal{N}}$ is irreducible and so $\mu^{-1}(\mathcal{N}_s)$ is dense.
\end{proof}

\begin{lemma}\label{birationality}
Let $GZ$ denote the orbit of all regular nilpotent elements. The morphism $\mu^{-1}(GZ) \to GZ$ is an isomorphism of varieties.
\end{lemma}

\begin{proof}
By \cite[Corollary 6.8]{JN}, $GZ$ is an open subset of $\mathcal{N}$, and $|\mu^{-1}(X)| = 1$ for $X \in GZ$. Hence $\mu$ induces a bijection $\mu^{-1}(GZ) \to GZ$. Since the morphism $\mu: T^*\mathcal{B} \to \mathcal{N}$ is given by projection onto the first coordinate, from Definition \ref{enhanced nilpotent cone}, it is a morphism of varieties and hence so is the restriction $\mu \mid_{\mu^{-1}(GZ)}: \mu^{-1}(GZ) \to GZ$. The result follows.
\end{proof}

Recall from Theorem \ref{polynomial theorem} that $S(\mathfrak{h})^W$ is a polynomial ring, with algebraically independent generators $f_1, \cdots, f_n$. \\

\begin{thm}\label{KW}
Let $G$ be a simple algebraic group, and suppose $(G,p) \neq (B,2)$. There is a projection map $\mathfrak{g} = \mathfrak{n}^- \oplus \mathfrak{h} \oplus \mathfrak{n} \to \mathfrak{h}$, which induces a map $S(\mathfrak{g}) \to S(\mathfrak{h})$. \\

This map induces a map $\eta: S(\mathfrak{g})^G \to S(\mathfrak{h})^W$, which is an isomorphism.
\end{thm}

\begin{proof}
This is \cite[Theorem 4]{KW}.
\end{proof}

When $G = PGL_n$ and $p|n$, the hypotheses of Theorem \ref{KW} are satisfied. This allows us to make sense of the following definition.

\begin{defn}\label{Steinberg quotient}
The $\emph{Steinberg quotient}$ is the map $\chi: \mathfrak{g} \to K^n$ defined by $\chi(Z) = (\eta^{-1}(f_1)(Z), \cdots, \eta^{-1}(f_n)(Z))$. Note that the nilpotent cone $\mathcal{N} = \chi^{-1}(0)$.
\end{defn} \vspace{4.4 mm}

\begin{lemma}\label{smooth elements}
The smooth points of $\mathcal{N}$ are precisely the regular nilpotent elements.
\end{lemma}

\begin{proof}
By the assumptions on the prime $p$, applying Theorem \ref{polynomial theorem} and Theorem \ref{freeness polynomial} shows that $S(\mathfrak{h})$ is a free $S(\mathfrak{h})^W$-module and $S(\mathfrak{h})^W$ is a polynomial ring, with generators $f_1, \cdots, f_n$. Hence the argument for \cite[Claim 6.7.10]{CG} applies and the Steinberg quotient $\chi$ satisfies, for $Z \in \mathfrak{g}$, the condition that $(d\chi)_Z$ is surjective if and only if $Z$ is regular. By \cite[Proposition 7.11]{JN}, for each $b = (b_1, \cdots, b_n) \in K^n$, the ideal of $\chi^{-1}(b)$ is generated by all $\eta^{-1}(f_i) - b_i$. \\

By \cite[I.5]{Ha}, $Z \in \chi^{-1}(b)$ is a smooth point if and only if the $d(\eta(f_i) - b_i)$ are linearly independent at $Z$ if and only if the map $(d\chi)_Z$ is surjective. Let $b=0$. Then the smooth points in $\chi^{-1}(0)$ are the regular elements contained in $\chi^{-1}(0)$, and so the smooth points of $\mathcal{N}$ are precisely the regular nilpotent elements. 
\end{proof}

\begin{thm}\label{Springer resolution}
$\mu: T^*\mathcal{B} \to \mathcal{N}$ is a resolution of singularities for $\mathcal{N}$.
\end{thm}

\begin{proof}
By Lemma \ref{enhanced nilpotent cone smooth} and Lemma \ref{enhanced nilpotent cone cotangent}, $\widetilde{\mathcal{N}}$ is a smooth irreducible variety. Furthermore, $\mu$ is proper by \cite[Lemma 6.10(1)]{JN}. By Lemma \ref{density}, $\mu^{-1}(\mathcal{N}_s)$ is dense in $\widetilde{\mathcal{N}}$, and by Lemma \ref{birationality}, $\mu$ is a birational morphism between $\mu^{-1}(\mathcal{N}_s)$ and $\mathcal{N}_s$. Hence $\mu$ is a resolution of singularities.
\end{proof}

\subsection{The dual nilpotent cone is a normal variety}\label{Geometric results}

In this section, we demonstrate that the dual nilcone $\mathcal{N}^*$ is a normal variety in the case $G = PGL_n$, $p|n$. 

\begin{defn}
Since we have a $G$-equivariant isomorphism $\kappa: \mathfrak{g} \to \mathfrak{g}^*$ by \cite[Section 6.5]{JN}, the $\emph{dual nilcone}$ $\mathcal{N}^*$ may be defined as:

\begin{gather*}
\mathcal{N}^* = \lbrace X \in \mathfrak{g}^* \mid f(X) = 0 \text{ for all } f \in S^+(\mathfrak{g})^G \rbrace.
\end{gather*}
\end{defn}

The same argument as in Proposition \ref{affine nilpotent cone} shows that $\mathcal{N} = V(S^+(\mathfrak{g})^G)$ is an affine variety. 

We next review some basic properties of normal rings and varieties. \\

\begin{defn} \cite[Definition 2.2.12]{CG}
A finitely generated commutative $K$-algebra $A$ is $\textit{Cohen-Macaulay}$ if it contains a subalgebra of the form $\mathcal{O}(V)$ such that $A$ is a free $\mathcal{O}(V)$-module of finite rank, and $V$ is a smooth affine scheme. \\

A scheme $X$ defined over $K$ is $\emph{Cohen-Macaulay}$ if, at each point $x \in X$, the local ring $\mathcal{O}_{X,x}$ is a Cohen-Macaulay ring.
\end{defn} \vspace{4.4 mm}

\begin{defn}
A commutative ring $A$ is $\textit{normal}$ if the localization $A_{\mathfrak{p}}$ for each prime ideal $\mathfrak{p}$ is an integrally closed domain. \\

A variety $V$ is $\textit{normal}$ if, for any $x \in V$, the local ring $\mathcal{O}_{V,x}$ is a normal ring. \\
\end{defn}

We now begin the proof of the normality of the dual nilpotent cone $\mathcal{N}^*$. We adapt the arguments in \cite{BL} to our situation. \\

\begin{thm}\label{important lemma}
Let $X$ be an irreducible affine Cohen-Macaulay scheme defined over $K$ and $U \subseteq X$ an open subscheme. Suppose $\text{dim } (X/U) \leq \text{dim } X - 2$ and that the scheme $U$ is normal. Then the scheme $X$ is normal.
\end{thm}

\begin{proof}
This is \cite[Corollary 2.3]{BL}.
\end{proof}

We aim to apply Theorem \ref{important lemma} to our situation. We begin with the following lemma, a variant on Hartogs' lemma. \\

\begin{lemma}\label{stackexchange lemma}
Let $Y$ be an affine normal variety and $Z \subseteq Y$ be a subvariety of codimension at least 2. Then any rational function on $Y$ which is regular on $Y \setminus Z$ can be extended to a regular function on $Y$.
\end{lemma}

\begin{proof}
Write $Y = \text{Spec } B$, where $B$ is a normal domain. Set $Z := V(I)$ for some ideal $I$, and write $U := Y \setminus Z$. Then $U = \bigcup_{f \in I} D(f)$, where $D(f)$ denotes the basic open sets in the Zariski topology. \\

Let $\mathfrak{p}$ be a prime ideal of height 1. By assumption, $\text{ht } I \geq 2$, and so there exists some $f \in I$ with $f \notin \mathfrak{p}$. It follows that $B_f \subseteq B_{\mathfrak{p}}$. \\

Let $a/b$ be a regular function on $U$, with $a/b \in \text{Frac} B$, the field of fractions of $B$. Since $\mathfrak{p}$ has height 1, we can find $f \in I \setminus \mathfrak{p}$. Then $a/b$ is regular on $D(f)$, and so $a/b \in \mathcal{O}(D(f)) = B_f \subseteq B_{\mathfrak{p}}$. As $\mathfrak{p}$ was arbitrary, $a/b \in \bigcap_{\text{ht}\mathfrak{p} = 1} B_{\mathfrak{p}} = B$. Hence $a/b$ can be extended to a regular function on $Y$.
\end{proof} 

\begin{lemma}\label{function field isomorphism}
Let $X$ be an affine Cohen-Macaulay scheme with an open subscheme $U$. Let $r: \mathcal{O}(X) \to \mathcal{O}(U)$ be the restriction morphism. Then: \\

(a) if $\text{dim }(X \setminus U) < \text{dim } X$, then $r$ is injective, \\
(b) if $\text{dim } (X \setminus U) \leq \text{dim } X - 2$, then $r$ is an isomorphism.
\end{lemma}

\begin{proof}
We expand on the proof given in \cite[Lemma 2.2]{BL}. For ease of notation, we suppose $\mathcal{O}(X)$ is a finitely generated $\mathcal{O}(Y)$-module for some smooth affine scheme $Y$. Now the projection map $p: X \to Y$ is a finite morphism and hence is closed. Without loss of generality, we can shrink $U$, replacing it by a smaller open subset $p^{-1}(W)$, where $W = Y \setminus p(X \setminus U)$ is an open subset of $Y$. \\

Let $F := p_*(\mathcal{O}_X)$. This is a free $\mathcal{O}_Y$-module and we clearly have $\Gamma(Y,F) = p_*(\mathcal{O}_X)(Y) = \mathcal{O}_X(p^{-1}(Y)) = \mathcal{O}_X(X)$, and similarly $\Gamma(W,F) = p_*(\mathcal{O}_X)(W) = \mathcal{O}_X(p^{-1}(W)) = \mathcal{O}_X(U)$. Hence the restriction morphism $r$ agrees with the natural restriction map $r: \Gamma(Y,F) \to \Gamma(W,F)$. \\

If $\text{dim }(X \setminus U) < \text{dim } X$, then $\text{dim }(p(U \setminus X)) < \text{dim } (p(X))$, so $\text{dim }(Y \setminus W) < \text{dim } Y$, and so $r$ is injective. \\

Similarly, if $\text{dim } (X \setminus U) \leq \text{dim } X - 2$, then $\text{dim } (Y \setminus W) \leq \text{dim } Y - 2$. Hence, by Lemma \ref{stackexchange lemma}, any regular function on $W$ can be extended to a regular function on $Y$. Furthermore, $F$ is a free $\mathcal{O}_Y$-module; it follows that $r$ is surjective.
\end{proof}

As an immediate consequence, we see that if the scheme $U$ is reduced and normal, then so is $X$. \\

We now demonstrate that the hypotheses in Theorem \ref{important lemma} are satisfied in our situation. Recall that $\mathcal{N}^*$ is an affine variety. It suffices to show that $\mathcal{N}^*$ is irreducible and Cohen-Macaulay. \\

\begin{defn}
$\lambda \in \mathfrak{h}^*$ is $\emph{regular}$ if its centraliser in $\mathfrak{g}$ under the natural $\mathfrak{g}$-action on $\mathfrak{g}^*$ coincides with the Cartan subalgebra $\mathfrak{h}$. A general $\lambda \in \mathfrak{g}^*$ is $\emph{regular}$ if its coadjoint orbit contains a regular element of $\mathfrak{h}^*$. 
\end{defn}

The subvariety $U$ in Lemma \ref{function field isomorphism} will be taken to be the subset of regular nilpotent elements. \\

\begin{prop} \label{codimension prop}
Suppose $p$ is nonspecial for $G$. Then:

(a) the dual nilcone $\mathcal{N}^* \subseteq \mathfrak{g}^*$ is a closed irreducible subvariety of $\mathfrak{g}^*$, and it has codimension $r$ in $G$, where $r$ is the rank of $G$. \\

(b) Let $U$ denote the set of regular elements of $\mathcal{N}^*$. Then $U$ is a single coadjoint orbit, which is open in $\mathcal{N}^*$, and its complement has codimension $\geq 2$.
\end{prop}

\begin{proof}
(a) We define an auxiliary variety $S$ via:

\begin{gather*}
S := \lbrace (gB, \zeta) \in G/B \times \mathfrak{g}^* \mid g \cdot \zeta \in \mathfrak{b}^{\perp} \rbrace.
\end{gather*}

This subset of $G/B \times \mathfrak{g}^*$ is closed. Define a map $\phi: G \times \mathfrak{b}^{\perp} \to G/B \times \mathfrak{g}^*$ by $\phi(g, \zeta) = (gB, g^{-1} \cdot \zeta)$. Now the image of $\phi$ is contained in $S$, and we can also see that $\text{im}(\phi) \cong S$ since we have a linear isomorphism $\mathfrak{b}^{\perp} \to g^{-1} \cdot \mathfrak{b}^{\perp}$. Hence the image of $\phi$ coincides with $S$. It follows that $S$ is a morphic image of an irreducible variety, and hence $S$ is itself an irreducible subvariety. \\

Let $p_1: G/B \times \mathfrak{g}^* \to G/B$ and $p_2: G/B \times \mathfrak{g}^* \to \mathfrak{g}^*$ be the obvious projection maps. Clearly $p_1(S) = G/B$. The fiber of $gB$ under the map $p_1$ is $g^{-1} \cdot \mathfrak{n}$, which is isomorphic to $\mathfrak{n}$. Hence the fibers are equidimensional, and we have:

\begin{gather*}
\text{dim } S = \text{dim } G/B + \text{dim } \mathfrak{n}, \\
\text{dim } S = \text{dim } G/B + \text{dim } U \\
= \text{dim } G - r.
\end{gather*}

Using the second projection, $\text{dim } p_2(S) \leq \text{dim } G - r$, with equality if some fibre is finite (as a set). First notice that:

\begin{gather*}
p_2(S) = \lbrace \zeta \in \mathfrak{g}^* \mid \exists g \in G \text{ s.t. } g \cdot \zeta \in \mathfrak{b}^{\perp} \rbrace = \mathcal{N}^*.
\end{gather*} 

Hence $\mathcal{N}^*$ is irreducible, and, since the flag variety $G/B$ is complete by \cite{Bo}, $\mathcal{N}^*$ is closed. We show that there exists some $\zeta \in \mathfrak{g}^*$ with:

\begin{gather*}
| \lbrace gB \mid g \cdot \zeta \in \mathfrak{b}^{\perp} \rbrace | < \infty, \\
| \lbrace gB \mid \zeta(\text{Ad}_g^{-1}(\mathfrak{b})) = 0 \rbrace | < \infty. 
\end{gather*} 

By \cite[Proposition 2]{HS}, we have the following dimension formula:

\begin{gather*}
\text{dim } p_1(p_2^{-1}(\zeta)) = \frac{\text{dim } Z_G(\zeta) - r}{2}.
\end{gather*}

Since $p$ is nonspecial for $G$, the set of regular nilpotent elements $U$ in $\mathcal{N}^*$ is non-empty, by \cite[Section 6.4]{GKM}, and thus we can always pick some $\zeta \in \mathfrak{g}^*$ such that $\text{dim } Z_G(\zeta) - r = 0$. Thus there exists $\zeta$ with $| \lbrace p_1(p_2^{-1}(\zeta)) \rbrace | < \infty$. \\

Now consider two points $(gB, \zeta), (hB, \zeta) \in S$. By definition, $g \cdot \zeta \in \mathfrak{b}^{\perp}$ and $h \cdot \zeta \in \mathfrak{b}^{\perp}$. The coadjoint action then gives $\zeta(\text{ad}_g^{-1}(\mathfrak{b})) = \zeta(\text{ad}_h^{-1}(\mathfrak{b})) = 0$. It follows that $gB = hB$, and so $p_1$ is injective when restricted to the fibre $p_2^{-1}(\zeta)$. It follows that there is a fibre of $p_2$ which is finite as a set. \\

Given the existence of a finite fibre of $p_2$, we have $\text{dim } S = \text{dim } p_2(S) = \text{dim } \mathcal{N}^* = \text{dim } G - r$. \\

(b) Now $\mathcal{N}^*$ has only finitely many $G$-orbits by \cite{Xu2} and \cite[Proposition 7.1]{Xu3}, so the dimension of $\mathcal{N^*}$ is equal to the dimension of at least one of these orbits. Since $\text{dim } \mathcal{N^*} = \text{dim } G - r$, some orbit in $\mathcal{N^*}$ also has dimension equal to $\text{dim } G - r$. This orbit is regular and its closure is all of $\mathcal{N^*}$, since the dimensions are equal and $\mathcal{N^*}$ is irreducible.  Since any $G$-orbit is open in its closure, by \cite[1.13, Corollary 1]{St3}, this class is open in $\mathcal{N^*}$ and thus is dense. \\

Let $R$ be the root system of $G$ and fix a subset of positive roots $R^+ \subseteq R$. Let $\alpha_i$ be a simple root, $X_{\alpha}$ the corresponding root subgroup, and set $U_i := \prod_{\alpha \in R^+, \alpha \neq \alpha_i} X_{\alpha}$. Let $T$ be the maximal torus of $G$ defined by this root system and let $P_i := T \cdot \langle X_{\alpha_i}, X_{-\alpha_i} \rangle \cdot U_i$. Since both $T$ and $\langle X_{\alpha_i}, X_{-\alpha_i} \rangle$ normalise $U_i$ by the commutation formulae in \cite[3.7]{St3}, we see that $P_i$ is a rank 1 parabolic subgroup of $G$, $U_i$ is its unipotent radical and $T \cdot \langle X_{\alpha_i}, X_{-\alpha_i} \rangle$ is a Levi subgroup of $P_i$. \\

Note that $\text{dim } T \cdot \langle X_{\alpha_i}, X_{-\alpha_i} \rangle = r+2$ and so $\text{dim } P_i - \text{dim } U_i = r+2$. \\

Parallel to the definition of the variety $S$, we set:

\begin{gather*}
S_i := \lbrace (gP_i, \zeta) \in G/P_i \times \mathfrak{g}^* \mid g \cdot \zeta \in \mathfrak{b}^{\perp}_i \rbrace
\end{gather*}

where $\mathfrak{b}^{\perp}_i = \lbrace \zeta \in \mathfrak{g}^* \mid \zeta(\text{Lie}(U_iT)) = 0 \rbrace$. Then $S_i$ is a closed and irreducible variety and, by the same argument as in part (a) of the proposition:

\begin{gather*}
\text{dim } S_i = \text{dim } G/P_i + \text{dim } U_i \\
= \text{dim } G - (r+2).
\end{gather*}

Projecting onto the second factor, we see that:

\begin{gather*}
\text{dim } p_2(S_i) \leq \text{dim } S_i = \text{dim } G - (r+2).
\end{gather*}

But an element $\zeta \in \mathcal{N}^*$ fails to be regular if and only if $G \cdot \zeta \cap \mathfrak{h}^*_{\text{reg}} = \emptyset$. By the decomposition in \cite[Section 6.4]{GKM}, this occurs precisely when the centraliser of each $\xi \in G \cdot \zeta \cap \mathfrak{h}^*$ contains some non-zero root $\alpha$ such that $\xi(\alpha^{\vee}(1)) = 0$, where $\alpha^{\vee}$ is the coroot corresponding to $\alpha$. It follows that $\zeta \in \mathcal{N}^*$ fails to be regular if and only if it lies in $p_2(S_i)$ for some $i$. Then:

\begin{gather*}
\text{dim }(\mathcal{N}^* \setminus U) = \text{sup}_i \text{ dim } p_2(S_i) \leq \text{dim } G - (r+2).
\end{gather*}
\end{proof}

\begin{lemma} \label{Kostant}
Let $r: S(\mathfrak{g}) \to S(\mathfrak{h})$ be the natural map, and $r^{\prime}$ its restriction to the graded subalgebra $S(\mathfrak{g})^G$. Suppose that $r^{\prime}$ is an isomorphism onto its image $S(\mathfrak{h})^W$ and $S(\mathfrak{h})$ is a free $S(\mathfrak{h})^W$-module. Then $S(\mathfrak{g})$ is a free $R$-module, where $R := S(\mathfrak{g}/\mathfrak{h}) \otimes S(\mathfrak{g})^G$, and hence is a free $S(\mathfrak{g})^G$-module.
\end{lemma}

\begin{proof}
The argument is similar to that which is set out in \cite[2.2.12]{CG} and the following discussion. Consider the projection map $\mathfrak{g} \to \mathfrak{g}/\mathfrak{h}$. This makes $\mathfrak{g}$ a vector bundle over $\mathfrak{g}/\mathfrak{h}$, and defines a natural increasing filtration on $S(\mathfrak{g})$ via:

\begin{gather*}
F_pS(\mathfrak{g}) = \lbrace P \in S(\mathfrak{g}) \mid P \text{ has degree } \leq p \text{ along the fibers} \rbrace.
\end{gather*}

Let $\text{gr}_F(S(\mathfrak{g}))$ denote the associated graded ring corresponding to this filtration, and set $S(\mathfrak{g})(p)$ to denote the $p$-th graded component. Clearly $S(\mathfrak{g})(0) = S(\mathfrak{g}/\mathfrak{h})$, and each graded component is an infinite-dimensional free $S(\mathfrak{g}/\mathfrak{h})$-module. There is a $K$-algebra isomorphism:

\begin{gather*}
S(\mathfrak{g})(p) \cong S(\mathfrak{g}/\mathfrak{h}) \otimes_K S^p(\mathfrak{h}),
\end{gather*}

where $S^p(\mathfrak{h})$ denotes the space of degree $p$ homogeneous polynomials on $\mathfrak{h}$. \\

Let $\sigma_p: F_pS(\mathfrak{g}) \to S(\mathfrak{g})(p)$ be the principal symbol map. Suppose $f \in F_pS(\mathfrak{g})$ is a homogeneous degree $p$ polynomial whose restriction $r(f)$ to $\mathfrak{h}$ is non-zero. Then $\sigma_p(f)$ equals the image of the element $1 \otimes_k r(f)$ under the above isomorphism, and so is non-zero in $S(\mathfrak{g})(p)$. \\

To see this, choose a vector subspace $\mathfrak{j}$ of $\mathfrak{g}$ such that $\mathfrak{g} = \mathfrak{h} \oplus \mathfrak{j}$. This yields a graded algebra isomorphism $S(\mathfrak{g}) = S(\mathfrak{h}) \otimes S(\mathfrak{j})$, and so one writes $F_pS(\mathfrak{g}) = \sum_{i \leq p} S^i(\mathfrak{h}) \otimes S^{p-i}(\mathfrak{j})$. Hence $f \in F_pS(\mathfrak{g})$ has the form:

\begin{gather*}
f = e_p \otimes 1 + \sum_{i \leq p} e_i \otimes w_{p-i},
\end{gather*}

where $e_i \in S^i(\mathfrak{h})$ and $w_{p-i} \in S^{p-i}(\mathfrak{j})$. Hence $r(f) = e_p$ and $\sigma_p(f) = e_p \otimes 1$, as required. \\

Given this claim, consider the filtration $F^pS(\mathfrak{g})^G$ in $S(\mathfrak{g})$. For any homogeneous element $f \in S(\mathfrak{g})^G$, its symbol $\sigma_p(f)$ coincides with $r(f) \in S(\mathfrak{h}) \subseteq \text{gr}_{\Phi} (S(\mathfrak{g}))$. Hence the subalgebra $\sigma_{\Phi}(f) \subseteq \text{gr}_{\Phi} (S(\mathfrak{g}))$ coincides with $r(S(\mathfrak{g})^G) = S(\mathfrak{h})^W$. \\

Let $\lbrace a_k \rbrace$ be a free basis for the $S(\mathfrak{h})^W$-module $S(\mathfrak{h})$, and fix $b_k \in S(\mathfrak{g})$ with $r(b_k) = a_k$. Then $\sigma_p(b_k) = a_k$. The $a_k$ form a free basis of the $\text{gr}_pR$-module $\text{gr}_p (S(\mathfrak{g})) = S(\mathfrak{h}) \otimes S(\mathfrak{g}/\mathfrak{h})$, via tensoring on the right and applying the second part of the claim. It follows that the $\lbrace b_k \rbrace$ form a free basis of the $R$-module $S(\mathfrak{g})$.
\end{proof}

\begin{thm}\label{main theorem 1}
Let $G = PGL_n$ and suppose $p|n$. Then the dual nilpotent cone $\mathcal{N}^* \subseteq \mathfrak{g}^*$ is a normal variety.
\end{thm}

\begin{proof}
Recall that $\mathcal{N}^*$ is an affine variety with defining ideal $J := V(S^+(\mathfrak{g})^G))$. It follows that its algebra of global functions $\mathcal{O}(\mathcal{N}^*) = S(\mathfrak{g})/J$. Consider $Y := \mathfrak{g}/\mathfrak{h}$ as an affine variety. Then Lemma $\ref{Kostant}$ implies that $\mathcal{O}(\mathcal{N}^*)$ is a free finitely generated module over the polynomial algebra $S(Y)$. Hence $\mathcal{N}^*$ is a Cohen-Macaulay variety. \\

By Proposition \ref{codimension prop}, $\mathcal{N}^*$ is a closed irreducible subvariety of $\mathfrak{g}^*$, and the complement of the set of regular elements $U$ in $\mathcal{N}^*$ has codimension $\geq 2$. Hence all conditions in the statement of Theorem \ref{important lemma} are satisfied, and so $\mathcal{N}^*$ is normal.
\end{proof}

$\textbf{Proof of Theorem A:}$ This is immediate from Theorem \ref{main theorem 1}. \\

We conclude this section with an application of this result, which will be used in later sections. \\

\begin{cor} \label{global sections isomorphism}
We have an isomorphism $\mu^*: \mathcal{O}(\mathcal{N}^*) \to \mathcal{O}(T^*\mathcal{B})$.
\end{cor}

\begin{proof}
The map $\mu: T^*\mathcal{B} \to \mathcal{N}$ is a resolution of singularities by Theorem \ref{Springer resolution}. Let $\tau: T^*\mathcal{B} \to \mathcal{N}^*$ be the composition of $\mu$ with the $G$-equivariant isomorphism $\kappa: \mathcal{N} \to \mathcal{N}^*$ from \cite[Section 6.5]{JN}. This induces an isomorphism $\tau^s: \tau^{-1}((\mathcal{N^*})^s) \to (\mathcal{N^*})^s$ on the smooth points. These are non-empty open subsets of $T^*\mathcal{B}$ and $\mathcal{N^*}$ respectively, and so $T^*\mathcal{B}$ and $\mathcal{N^*}$ are birationally equivalent. \\

Let $\mathcal{Q}(A)$ denote the field of fractions of an integral domain $A$. By \cite[I.4.5]{Ha}, $\mu$ induces an isomorphism $\mathcal{Q}(\mathcal{O}(\mathcal{N^*})) \to \mathcal{Q}(\mathcal{O}(T^*\mathcal{B}))$, and so $\mathcal{O}(T^*\mathcal{B})$ can be considered as a subring of $\mathcal{Q}(\mathcal{O}(\mathcal{N^*}))$. \\

Since the map $T^*\mathcal{B} \to \mathcal{N^*}$ is surjective, and $\mathcal{O}(T^*\mathcal{B})$, $\mathcal{O}(\mathcal{N^*})$ are integral domains, there is an inclusion $\mathcal{O}(\mathcal{N^*}) \to \mathcal{O}(T^*\mathcal{B})$. The map $\tau$ is proper, and so the direct image sheaf $\tau_*\mathcal{O}_{T^*\mathcal{B}}$ is a coherent $\mathcal{O}_{\mathcal{N^*}}$-module. In particular, taking global sections, we have that $\Gamma(\mathcal{N^*}, \tau_*\mathcal{O}_{T^*\mathcal{B}})$ is a finitely generated $\mathcal{O}(\mathcal{N^*})$-module. By definition, $\Gamma(\mathcal{N^*}, \tau_*\mathcal{O}_{T^*\mathcal{B}}) = \mathcal{O}(T^*\mathcal{B})$, so $\mathcal{O}(T^*\mathcal{B})$ is a finitely generated $\mathcal{O}(\mathcal{N^*})$-module. \\

The variety $\mathcal{N}^*$ is normal, and so $\mathcal{O}(\mathcal{N}^*)$ is an integrally closed domain. Let $b \in \mathcal{O}(T^*\mathcal{B})$. Then clearly $\mathcal{O}(T^*\mathcal{B})b \subseteq \mathcal{O}(T^*\mathcal{B})$, and hence $b$ is integral over $\mathcal{O}(\mathcal{N}^*)$. Hence, by integral closure, $b \in \mathcal{N}^*$ and there is an isomorphism $\mathcal{O}(\mathcal{N}^*) \to \mathcal{O}(T^*\mathcal{B})$.
\end{proof}

\subsection{Analogous results when $G$ is not of type A}

The restriction that $G = PGL_n$, $p|n$ plays a role in only a few places in the argument that $\mathcal{N}^*$ is a normal variety. In this section, we indicate some of the issues that arise when we replace $PGL_n$ by a more general simple algebraic group of adjoint type. \\

Theorem \ref{polynomial theorem} demonstrated that, in case $G = PGL_n$, $p|n$, the Weyl group invariants $S(\mathfrak{h})^W$ is a polynomial ring. This result is usually false in bad characteristic. In case the $W$-action on $\mathfrak{h}$ is irreducible,\cite[Theorem 3]{Bro} gives a full classification of the types in which this result holds, drawing on \cite[Theorem 7.2]{KM}. \\

\begin{prop}\label{fulton harris}
Suppose the pair (Dynkin diagram of $G$, $p$) lies in the following list: \\

(a) ($E_7$, 3), \\
(b) ($E_8$, 2), \\
(c) ($E_8$, 3), \\
(d) ($E_8$, 5), \\
(e) ($F_4$, 3), \\
(f) ($G_2$, 2). \\

Then the $W$-action on $\mathfrak{h}$ is irreducible.
\end{prop}

\begin{proof}
In all of these cases, the argument in \cite[Section 6.5]{JN} demonstrates that there is a $G$-equivariant bijection $\kappa: \mathfrak{g} \to \mathfrak{g}^*$, which restricts to a $G$-equivariant bijection $\mathfrak{h} \to \mathfrak{h}^*$. Furthermore, the classification in \cite[Section 0.13]{Hu2} demonstrates that $\mathfrak{g}$ is simple. Given these two statements, we may apply the same proof as that given in \cite[Proposition 14.31]{FH} to obtain the result.
\end{proof}

\begin{thm}\label{G2 polynomial}
Suppose $G$ is of type $G_2$ and $p = 2$. Then the invariant ring $S(\mathfrak{h})^W$ is polynomial.
\end{thm}

\begin{proof}
This follows from the calculations in \cite[Theorem 7.2]{KM}.
\end{proof}

In case $G$ is of type $G_2$ and $p = 2$, we may apply the same argument as for $G = PGL_n$ to obtain the following result. \\

\begin{thm}\label{G2 case}
Let $(G,p) = (G_2, 2)$. Then the dual nilpotent cone $\mathcal{N}^* \subseteq \mathfrak{g}^*$ is a normal variety.
\end{thm}

$\textbf{Proof of Theorem B:}$ This is immediate from Theorem \ref{G2 case}. \\

If $S(\mathfrak{h})^W$ is not polynomial, there are significant obstacles to generalising the result that $\mathcal{N}^*$ is a normal variety. In particular, the following behaviour may be observed. \\

- Kostant's freeness theorem, stated as Theorem \ref{freeness polynomial}, fails. This means that $S(\mathfrak{h})$ is not free as an $S(\mathfrak{h})^W$-module, meaning that we cannot apply the argument in Lemma \ref{Kostant} to show that $\mathcal{N^*}$ is a Cohen-Macaulay variety. \\

- The Steinberg quotient $\chi: \mathfrak{g} \to K^n$, defined in Definition \ref{Steinberg quotient}, makes sense as an abstract function, but since the generators $\lbrace f_1, \cdots, f_n \rbrace$ of $S(\mathfrak{h})^W$ are not algebraically independent, we cannot apply the argument in Lemma \ref{smooth elements} to show that the smooth elements of $\mathcal{N}$ coincide with the regular elements, which is a key step in the proof that the Springer resolution $\mu: T^*\mathcal{B} \to \mathcal{N}$ is a resolution of singularities for $\mathcal{N}$. \\

Calculations in \cite[Section 3.2]{Bro} show that, in the following cases (Dynkin diagram of $G$, $p$), the invariant ring $S(\mathfrak{h})^W$ is not even Cohen-Macaulay. \\

(a) $(E_7, 3)$, \\
(b) $(E_8, 3)$, \\
(c) $(E_8, 5)$. \\

\begin{conj}
In case the invariant ring $S(\mathfrak{h})^W$ is not Cohen-Macaulay, is it true that the dual nilpotent cone $\mathcal{N}^*$ is not a normal variety?
\end{conj}

\newpage

\section{Applications to representations of $p$-adic Lie groups}\label{annals chapter}

\subsection{Generalising the Beilinson-Bernstein theorem for $\widehat{\mathcal{D}^{\lambda}_{n,K}}$}

In this section, we apply the results of Section \ref{nilcone chapter} to the constructions given in \cite{AW}. This allows us to weaken the restrictions on the characteristic of the base field given in \cite[Section 6.8]{AW}, thereby providing us with generalisations of their results. \\

Throughout Section \ref{annals chapter}, we suppose $R$ is a fixed complete discrete valuation ring with uniformiser $\pi$, residue field $k$ and field of fractions $K$. Assume throughout this section that $K$ has characteristic 0 and $k$ is algebraically closed. \\

We recall some of the arguments from \cite[Section 4]{AW}, to define the sheaf of enhanced vector fields $\widetilde{\mathcal{T}}$ on a smooth scheme $X$, and the relative enveloping algebra $\widetilde{\mathcal{D}}$ of an $\textbf{H}$-torsor $\xi: \widetilde{X} \to X$. \\

Let $X$ be a smooth separated $R$-scheme that is locally of finite type. Let $\textbf{H}$ be a flat affine algebraic group defined over $R$ of finite type, and let $\widetilde{X}$ be a scheme equipped with an $\textbf{H}$-action. \\
 
\begin{defn}\label{torsor def}
A morphism $\xi: \widetilde{X} \to X$ is an $\textbf{H}-\emph{torsor}$ if: \\

(i) $\xi$ is faithfully flat and locally of finite type, \\
(ii) the action of $\textbf{H}$ respects $\xi$, \\
(iii) the map $\widetilde{X} \times \textbf{H} \to \widetilde{X} \times_X \widetilde{X}$, $(x, h) \to (x, hx)$ is an isomorphism. \\

An open subscheme $U$ of $X$ $\emph{trivialises the torsor}$ $\xi$ if there is an $\textbf{H}$-invariant isomorphism:

\begin{gather*}
U \times \textbf{H} \to \xi^{-1}(U)
\end{gather*}

where $\textbf{H}$ acts on $U \times \textbf{H}$ by left translation on the second factor.
\end{defn} \vspace{4.4 mm}

\begin{defn}\label{locally trivial defn}
Let $\mathcal{S}_X$ denote the set of open subschemes $U$ of $X$ such that: \\

(i) $U$ is affine, \\
(ii) $U$ trivialises $\xi$, \\
(iii) $\mathcal{O}(U)$ is a finitely generated $R$-algebra. \\

$\xi$ is $\emph{locally trivial}$ for the Zariski topology if $X$ can be covered by open sets in $\mathcal{S}_X$.
\end{defn} \vspace{4.4 mm}

\begin{lemma}\label{locally trivial lemma}
If $\xi$ is locally trivial, then $\mathcal{S}_X$ is a base for $X$.
\end{lemma}

\begin{proof}
Since $X$ is separated, $\mathcal{S}_X$ is stable under intersections. If $U \in \mathcal{S}_X$ and $W$ is an open affine subscheme of $U$, then $W \in \mathcal{S}_X$. Hence $\mathcal{S}_X$ is a base for $X$.
\end{proof}

The action of $\textbf{H}$ on $\widetilde{X}$ induces a rational action of $\textbf{H}$ on $\mathcal{O}(V)$ for any $\textbf{H}$-stable open subscheme $V \subseteq \widetilde{X}$, and therefore induces an action of $\textbf{H}$ on $\mathcal{T}_{\widetilde{X}}$ via:

\begin{gather*}
(h \cdot \partial)(f) = h \cdot \partial(h^{-1} \cdot f)
\end{gather*}

for $\partial \in \mathcal{T}_{\widetilde{X}}, f \in \mathcal{O}(\widetilde{X})$ and $h \in \textbf{H}$. The $\emph{sheaf of enhanced vector fields}$ on $X$ is:

\begin{gather*}
\widetilde{\mathcal{T}} := (\xi_*\mathcal{T}_{\widetilde{X}})^{\textbf{H}}.
\end{gather*}

Differentiating the $\textbf{H}$-action on $\widetilde{X}$ gives an $R$-linear Lie algebra homomorphism:

\begin{gather*}
j: \mathfrak{h} \to \mathcal{T}_{\widetilde{X}}
\end{gather*}

where $\mathfrak{h}$ is the Lie algebra of $\textbf{H}$. \\

\begin{defn}\label{enhanced cotangent bundle}
Let $\xi: \widetilde{X} \to X$ be an $\textbf{H}$-torsor. Then $\xi_*\mathcal{D}_{\widetilde{X}}$ is a sheaf of $R$-algebras with an $\textbf{H}$-action. The $\emph{relative enveloping algebra}$ of the torsor is the sheaf of $\textbf{H}$-invariants of $\xi_*\mathcal{D}_{\widetilde{X}}$:

\begin{gather*}
\widetilde{\mathcal{D}} := (\xi_*\mathcal{D}_{\widetilde{X}})^{\textbf{H}}.
\end{gather*}
\end{defn}

This sheaf has a natural filtration:

\begin{gather*}
F_m\widetilde{\mathcal{D}} := (\xi_*F_m\mathcal{D}_{\widetilde{X}})^{\textbf{H}}
\end{gather*}

induced by the filtration on $\mathcal{D}_{\widetilde{X}}$ by order of differential operator. \\

Let $\textbf{B}$ be a Borel subgroup of $\textbf{G}$. Let $\textbf{N}$ be the unipotent radical of $\textbf{B}$, and $\textbf{H} := \textbf{B}/\textbf{N}$ the abstract Cartan group. Let $\widetilde{\mathcal{B}}$ denote the homogeneous space $\textbf{G}/\textbf{N}$. There is an $\textbf{H}$-action on $\widetilde{\mathcal{B}}$ defined by:

\begin{gather*}
b\textbf{N} \cdot g \textbf{N} := gb\textbf{N}
\end{gather*}

which is well-defined since $[\textbf{B}, \textbf{B}]$ is contained in $\textbf{N}$. $\mathcal{B} := \textbf{G}/\textbf{B}$ is the $\emph{flag variety}$ of $\textbf{G}$. $\widetilde{\mathcal{B}}$ is the $\emph{basic affine space}$ of $\textbf{G}$. \\

By the splitting assumption of $\textbf{G}$, we can find a Cartan subgroup $\textbf{T}$ of $\textbf{G}$ complementary to $\textbf{N}$ in $\textbf{B}$. This is naturally isomorphic to $\textbf{H}$, and induces an isomorphism of the corresponding Lie algebras $\mathfrak{t} \to \mathfrak{h}$. \\

We let $\textbf{W}$ denote the Weyl group of $\textbf{G}$, and let $\textbf{W}_k$ denote the Weyl group of $\textbf{G}_k$, the $k$-points of the algebraic group $\textbf{G}$. \\

We may differentiate the natural $\textbf{G}$-action on $\widetilde{\mathcal{B}}$ to obtain an $R$-linear Lie homomorphism:

\begin{gather*}
\varphi: \mathfrak{g} \to \mathcal{T}_{\widetilde{\mathcal{B}}}.
\end{gather*}

Since the $\textbf{G}$-action commutes with the $\textbf{H}$-action on $\widetilde{\mathcal{B}}$, this map descends to an $R$-linear Lie homomorphism $\varphi: \mathfrak{g} \to \widetilde{\mathcal{T}}_{\mathcal{B}}$ and an $\mathcal{O}_{\mathcal{B}}$-linear morphism:

\begin{gather*}
\varphi: \mathcal{O}_{\mathcal{B}} \otimes \mathfrak{g} \to \widetilde{\mathcal{T}}_{\mathcal{B}}
\end{gather*}

of locally free sheaves on $\mathcal{B}$. Dualising, we obtain a morphism of vector bundles over $\mathcal{B}$:

\begin{gather*}
\varphi^*: \widetilde{T^*\mathcal{B}} \to \mathcal{B} \times \mathfrak{g}^*
\end{gather*}

from the enhanced cotangent bundle to the trivial vector bundle of rank dim $\mathfrak{g}$. \\

\begin{defn}\label{enhanced moment map}
The $\emph{enhanced moment map}$ is the composition of $\varphi^*$ with the projection onto the second coordinate:

\begin{gather*}
\beta: \widetilde{T^*\mathcal{B}} \to \mathfrak{g}^*.
\end{gather*}
\end{defn}

We may apply the deformation functor (\cite[Section 3.5]{AW}) to the map $j: U(\mathfrak{h}) \to \widetilde{\mathcal{D}}$, defined above Definition \ref{enhanced cotangent bundle}, to obtain a central embedding of the constant sheaf $U(\mathfrak{h})_n$ into $\widetilde{\mathcal{D}}_n$. This gives $\widetilde{\mathcal{D}}_n$ the structure of a $U(\mathfrak{h})_n$-module. \\

Let $\lambda \in \text{Hom}_R(\pi^n\mathfrak{h}, R)$ be a linear functional. This extends to an $R$-algebra homomorphism $U(\mathfrak{h})_n \to R$, which gives $R$ the structure of a $U(\mathfrak{h})_n$-module, denoted $R_{\lambda}$. \\

\begin{defn}
The $\emph{sheaf of deformed twisted differential operators}$ $\mathcal{D}^{\lambda}_n$ on $\mathcal{B}$ is the sheaf:

\begin{gather*}
\mathcal{D}^{\lambda}_n := \widetilde{\mathcal{D}_n} \otimes_{U(\mathfrak{h})_n} R_{\lambda}
\end{gather*}
\end{defn}

By \cite[Lemma 6.4(b)]{AW}, this is a sheaf of deformable $R$-algebras. \\ 

\begin{defn}
The $\pi$-$\emph{adic completion}$ of $\mathcal{D}^{\lambda}_n$ is $\widehat{\mathcal{D}^{\lambda}_n} := \varprojlim \mathcal{D}^{\lambda}_n /\pi^a\mathcal{D}^{\lambda}_n$. Furthermore, set $\widehat{\mathcal{D}^{\lambda}_{n,K}} := \widehat{\mathcal{D}^{\lambda}_n} \otimes_R K$.
\end{defn} \vspace{4.4 mm}

The adjoint action of $\textbf{G}$ on $\mathfrak{g}$ extends to an action on $U(\mathfrak{g})$ by ring automorphisms, which is filtration-preserving and so descends to an action on $\text{gr } U(\mathfrak{g}) \cong S(\mathfrak{g})$. Let:

\begin{gather*}
\psi: S(\mathfrak{g})^{\textbf{G}} \to S(\mathfrak{t})
\end{gather*}

denote the composition of the inclusion $S(\mathfrak{g})^{\textbf{G}} \to S(\mathfrak{g})$ with the projection $S(\mathfrak{g}) \to S(\mathfrak{t})$. By \cite[Theorem 7.3.7]{Di}, the image of $\psi$ is contained in $S(\mathfrak{t})^{\textbf{W}}$, and $\psi$ is injective. \\

Since taking $\textbf{G}$-invariants is left exact, we have an inclusion $\text{gr }(U(\mathfrak{g})^{\textbf{G}}) \to S(\mathfrak{g})^{\textbf{G}}$. Our next proposition gives a description of the associated graded ring of $U(\mathfrak{g})^{\textbf{G}}$. \\

\begin{prop}\label{6.9}
The rows of the diagram:

\begin{figure}[H]
\centering
\begin{tikzcd}
0 \arrow{r} & \text{gr}(U(\mathfrak{g})^{\textbf{G}}) \arrow{r}{\pi} \arrow{d}{\iota} & \text{gr}(U(\mathfrak{g})^{\textbf{G}}) \arrow{r} \arrow {d}{\iota} & \text{gr}(U(\mathfrak{g}_k)^{\textbf{G}_k}) \arrow{d}{\iota_k} \arrow{r} & 0 \\
0 \arrow{r} & S(\mathfrak{g})^{\textbf{G}} \arrow{r}{\pi} \arrow{d}{\psi} & S(\mathfrak{g})^{\textbf{G}} \arrow{r} \arrow{d}{\psi} & S(\mathfrak{g}_k)^{\textbf{G}_k} \arrow{d}{\psi_k} \arrow{r} & 0 \\
0 \arrow{r} & S(\mathfrak{t})^{\textbf{W}} \arrow{r}{\pi} & S(\mathfrak{t})^{\textbf{W}} \arrow{r} & S(\mathfrak{t}_k)^{\textbf{W}_k} \arrow{r} & 0 \\
\end{tikzcd}
\end{figure}

are exact, and each vertical map is an isomorphism.
\end{prop}

\begin{proof}
View the diagram as a sequence of complexes $C^{\bullet} \to D^{\bullet} \to E^{\bullet}$. Since $\pi$ generates the maximal ideal $\mathfrak{m}$ of $R$ by definition, and $R/\mathfrak{m} = k$, it is clear that each complex is exact in the left and in the middle. The exactness of $E^{\bullet}$ follows from the fact that $S(\mathfrak{t}_k)^{\textbf{W}_k}$ is a polynomial ring by Theorem \ref{polynomial theorem}: since $n > 2$ we may fix homogeneous generators $s_1, \cdots, s_l$ and lift these generators to homogeneous generators $S_1, \cdots, S_l$ of the ring $S(\mathfrak{t})^{\textbf{W}}$ with $s_i = S_i (\text{mod } \mathfrak{m})$ by the proof of \cite[Proposition 5.1]{KM}. Hence the map $S(\mathfrak{t})^{\textbf{W}} \to S(\mathfrak{t}_k)^{\textbf{W}_k}$ is surjective, and the complex $E^{\bullet}$ is exact. \\

By \cite[Theorem 7.3.7]{Di}, $\psi$ is injective, and since $p$ is nonspecial from Definition \ref{nonspecial defn}, $\psi_k$ is an isomorphism by Theorem \ref{KW}. Thus the composite map of complexes $\psi^{\bullet} \circ \iota^{\bullet}$ is injective. Set $F^{\bullet} := \text{coker}(\psi^{\bullet} \circ \iota^{\bullet})$: by definition, the sequence of complexes $0 \to C^{\bullet} \to E^{\bullet} \to F^{\bullet} \to 0$ is exact. \\

Since $C^{\bullet}$ is exact in the left and in the middle, $H^0(C^{\bullet}) = H^1(C^{\bullet}) = 0$. As $E^{\bullet}$ is exact, taking the long exact sequence of cohomology shows that $H^0(F^{\bullet}) = H^2(F^{\bullet}) = 0$ and yields an isomorphism $H^1(F^{\bullet}) \cong H^2(C^{\bullet})$. \\

Since $K$ is a field of characteristic zero, the map $\psi_K \circ \iota_K: \text{gr}(U(\mathfrak{g}_K)^{\textbf{G}_K}) \to S(\mathfrak{t}_K)^{\textbf{W}_K}$ is an isomorphism by \cite[Theorem 7.3.7]{Di}. Hence $F^0 = F^1 = \text{coker}(\psi \circ \iota)$ is $\pi$-torsion. Now $H^0(F^{\bullet}) = 0$, and so we have an exact sequence $0 \to F^0 \to F^1$. So $F^0 = F^1 = 0$, and hence $H^1(F^{\bullet}) = H^2(C^{\bullet}) = 0$. It follows that the top row $C^{\bullet}$ is exact. \\

Hence $\psi^{\bullet} \circ \iota^{\bullet}: C^{\bullet} \to E^{\bullet}$ is an isomorphism in all degrees except possibly 2, and so is an isomorphism via the Five Lemma. The result follows from the fact that $\psi^{\bullet}$ and $\iota^{\bullet}$ are both injections.
\end{proof}

It follows that, since $\psi \circ \iota$ is a graded isomorphism and $p$ is nonspecial, $\text{gr}(U(\mathfrak{g})^\textbf{G})$ is isomorphic to a commutative polynomial algebra over $R$ in $l$ variables by Theorem \ref{polynomial theorem}. The commutative polynomial algebra $R[x_1, \cdots, x_l]$ is a free $R$-module and hence is flat, and so $(U(\mathfrak{g})^\textbf{G})$ is a deformable $R$-algebra by \cite[Definition 3.5]{AW}. Furthermore, $\widehat{U(\mathfrak{g})^\textbf{G}_{n,K}}$ is also a commutative polynomial algebra over $R$ in $l$ variables, so the $\pi$-adic completion $\widehat{U(\mathfrak{g})^\textbf{G}_{n,K}}$ is a commutative Tate algebra. \\

By \cite[Proposition 4.10]{AW}, we have a commutative square consisting of deformable $R$-algebras:

\begin{figure}[H]
\centering
\begin{tikzcd}
U((\mathfrak{g})^{\textbf{G}})_n \arrow{r}{\phi_n} \arrow{d}{i_n} & U(\mathfrak{t})_n \arrow{d}{(j \circ i)_n} \\
U(\mathfrak{g})_n \arrow{r}{U(\phi)_n} & \widetilde{\mathcal{D}_n},
\end{tikzcd}
\end{figure}

We set:

\begin{gather*}
\mathcal{U}^{\lambda}_n := U(\mathfrak{g}) \otimes_{(U(\mathfrak{g})^{\textbf{G}})_n} R_{\lambda}, \\
\widehat{\mathcal{U}^{\lambda}_n} := \varprojlim \frac{\mathcal{U}^{\lambda}_n}{\pi^a \mathcal{U}^{\lambda}_n}, \\
\widehat{\mathcal{U}^{\lambda}_{n,K}} := \widehat{\mathcal{U}^{\lambda}_n} \otimes_R K.
\end{gather*}

By commutativity of the diagram, the map:

\begin{gather*}
U(\phi)_n \otimes (j \circ i)_n: U(\mathfrak{g})_n \otimes U(\mathfrak{t})_n \to \widetilde{\mathcal{D}_n}
\end{gather*}

factors through $U((\mathfrak{g})^\textbf{G})_n$, and we obtain the algebra homomorphisms:

\begin{gather*}
\phi^{\lambda}_n: \mathcal{U}^{\lambda}_n \to \mathcal{D}^{\lambda}_n, \\
\widehat{\phi^{\lambda}_n}: \widehat{\mathcal{U}^{\lambda}_n} \to \widehat{\mathcal{D}^{\lambda}_n}, \\
\widehat{\phi^{\lambda}_{n,K}}: \widehat{\mathcal{U}^{\lambda}_{n,K}} \to \widehat{\mathcal{D}^{\lambda}_{n,K}}.
\end{gather*}

\begin{thm}\label{6.10}
(a) $\widehat{\mathcal{U}^{\lambda}_{n,K}} \cong \widehat{U(\mathfrak{g})_{n,K}} \otimes_{\widehat{U(\mathfrak{g})^\textbf{G}_{n,K}}} K_{\lambda}$ is an almost commutative affinoid $K$-algebra. \\

(b) The map $\widehat{\phi^{\lambda}_{n,K}}: \widehat{\mathcal{U}^{\lambda}_{n,K}} \to \Gamma(\mathcal{B}, \widehat{\mathcal{D}^{\lambda}_{n,K}})$ is an isomorphism of complete doubly filtered $K$-algebras. \\

(c) There is an isomorphism $S(\mathfrak{g}_k) \otimes_{S(\mathfrak{g}_k)^{\textbf{G}_k}} k \cong \text{Gr }(\widehat{\mathcal{U}^{\lambda}_{n,K}})$.
\end{thm}

\begin{proof}
(a): This is identical to the proof given in \cite[Theorem 6.10(a)]{AW}. \\

(b): Let $ \lbrace U_1, \cdots, U_m \rbrace$ be an open cover of $\mathcal{B}$ by open affines that trivialise the torsor $\xi$, which exists by \cite[Lemma 4.7(c)]{AW}. The special fibre $\mathcal{B}_k$ is covered by the special fibres $U_{i,k}$. It suffices to show that the complex:

\begin{gather*}
C^{\bullet}: 0 \to \widehat{\mathcal{U}_{n,K}} \to \bigoplus_{i=1}^m \widehat{\mathcal{D}^{\lambda}_{n,K}}(U_i) \to \bigoplus_{i<j} \widehat{\mathcal{D}^{\lambda}_{n,K}}(U_i \cap U_j)
\end{gather*}

is exact. \\

Clearly, $C^{\bullet}$ is a complex in the category of complete doubly-filtered $K$-algebras, and so it suffices to show that the associated graded complex $\text{Gr}(C^{\bullet})$ is exact. By \cite[Corollary 3.7]{AW}, there is a commutative diagram with exact rows:

\begin{figure}[ht!]
\centering
\begin{tikzcd}
0 \arrow{r} & \text{gr}(U(\mathfrak{g})^{\textbf{G}}) \arrow{r}{\pi} \arrow{d} & \text{gr}(U(\mathfrak{g})^{\textbf{G}}) \arrow{r} \arrow {d} & \text{Gr}(\widehat{U(\mathfrak{g})^{\textbf{G}}_K}) \arrow{d} \arrow{r} & 0 \\
0 \arrow{r} & \text{gr}(U(\mathfrak{g})) \arrow{r}{\pi} & \text{gr}(U(\mathfrak{g})) \arrow{r} & \text{Gr}(\widehat{U(\mathfrak{g})_{n,K}}) \arrow{r} & 0. \\
\end{tikzcd}
\end{figure}

Via the identification $\text{gr}(U(\mathfrak{g})) = S(\mathfrak{g})$, Proposition \ref{6.9} induces a commutative square:

\begin{figure}[ht!]
\centering
\begin{tikzcd}
\text{Gr}(\widehat{U(\mathfrak{g})^{\textbf{G}}_{n,K}}) \arrow{r} \arrow{d} & S(\mathfrak{g}_k)^{\textbf{G}_k} \arrow{d} \\
\text{Gr}(\widehat{U(\mathfrak{g})_{n,K}}) \arrow{r} & S(\mathfrak{g}_k)
\end{tikzcd}
\end{figure}

where the horizontal maps are isomorphisms and the vertical maps are inclusions. Since $\text{Gr}(K_{\lambda})$ is the trivial $\text{Gr}(\widehat{U(\mathfrak{g})^{\textbf{G}}_{n,K}})$-module $k$, we have a natural surjection:

\begin{gather*}
S(\mathfrak{g}_k) \otimes_{S(\mathfrak{g}_k)^{{\textbf{G}}_k}} k \cong \text{Gr}(\widehat{U(\mathfrak{g})_{n,K}} \otimes_{\text{Gr}(\widehat{U(\mathfrak{g})^{\textbf{G}}_{n,K}})} \text{Gr}(K_{\lambda})) \to \text{Gr}(\widehat{\mathcal{U}^{\lambda}_{n,K}}).
\end{gather*}

This surjection fits into the commutative diagram:

\begin{figure}[ht!]
\centering
\begin{tikzcd}
0 \arrow{r} & S(\mathfrak{g}_k) \otimes_{S(\mathfrak{g}_k)^{{\textbf{G}}_k}} k \arrow{r} \arrow{d} & \bigoplus_{i=1}^m \mathcal{O}(T^*U_{i,k}) \arrow{r} & \bigoplus_{i<j} \mathcal{O}(T^*(U_{i,k} \cap U_{j,k})) \\
0 \arrow{r} & \text{Gr}(\widehat{\mathcal{U}^{\lambda}_{n,K}}) \arrow{r} & \bigoplus_{i=1}^m \text{Gr}(\widehat{\mathcal{D}^{\lambda}_{n,K}}(U_i)) \arrow{r} \arrow{u} & \bigoplus_{i<j} \text{Gr}(\widehat{\mathcal{D}^{\lambda}_{n,K}}(U_i \cap U_j)). \arrow{u} \\
\end{tikzcd}
\end{figure}

The bottom row is $\text{Gr}(C^{\bullet})$ by definition, and the top row is induced by the moment map $T^*\mathcal{B}_k \to \mathfrak{g}_k^*$. To see this, note that by Lemma \ref{enhanced nilpotent cone cotangent}, we have an identification $\widetilde{\mathcal{N}^*} \to T^*\mathcal{B}$ under our assumptions on $p$, and so exactness of the top row is equivalent to the existence of an isomorphism:

\begin{gather*}
S(\mathfrak{g}) \otimes_{S(\mathfrak{g})^{\textbf{G}}} k \cong \Gamma(\widetilde{\mathcal{N}^*}, \mathcal{O}_{\widetilde{\mathcal{N}^*}}).
\end{gather*}

By Theorem \ref{main theorem 1}, $\mathcal{N}^*$ is a normal variety and, by Theorem \ref{Springer resolution}, the map $\gamma: T^*\mathcal{B} \to \mathcal{N}^*$ is a resolution of singularities. It follows, by Corollary \ref{global sections isomorphism}, that there is an isomorphism of global sections:

\begin{gather*}
\gamma^*: \Gamma(\mathcal{N}^*, \mathcal{O}_{\mathcal{N}^*}) \to \Gamma(T^*\mathcal{B}, \mathcal{O}_{T^*\mathcal{B}}).
\end{gather*}

Recall from the proof of Theorem \ref{main theorem 1} that $\mathcal{O}(\mathcal{N}^*) = S(\mathfrak{g}) \otimes_{S(\mathfrak{g})^{\textbf{G}}} k$. Putting these isomorphisms together, we see that $S(\mathfrak{g}) \otimes_{S(\mathfrak{g})^{\textbf{G}}} k \cong \Gamma(\widetilde{\mathcal{N}^*}, \mathcal{O}_{\widetilde{\mathcal{N}^*}})$. \\

Now the second and third vertical arrows are isomorphisms by \cite[Proposition 6.5(a)]{AW}, which shows that $\text{Gr}(C^{\bullet})$ is exact. \\

(c) This is immediate, since one can also show that the first vertical arrow in the above diagram is an isomorphism via the Five Lemma.
\end{proof}

\begin{defn}\label{twisted localisation}
For each $\lambda \in \text{Hom}_R(\pi^n\mathfrak{h}, R)$, we define a functor:

\begin{gather*}
\text{Loc}^{\lambda}: \widehat{U(\mathfrak{g})^{\lambda}_{n,K}}-\text{mod} \to \widehat{\mathcal{D}^{\lambda}_{n,K}}-\text{mod}
\end{gather*}

given by $M \mapsto \widehat{\mathcal{D}^{\lambda}_{n,K}} \otimes_{\widehat{\mathcal{U}^{\lambda}_{n,K}}} M$.
\end{defn}

\subsection{Modules over completed enveloping algebras}

The adjoint action of $\textbf{G}$ on $\mathfrak{g}$ induces an action of $\textbf{G}$ on $U(\mathfrak{g})$ by algebra automorphisms. Composing the inclusion $U(\mathfrak{g})^{\textbf{G}} \to U(\mathfrak{g})$ with the projection $U(\mathfrak{g}) \to U(\mathfrak{t})$ defined by the direct sum decomposition $\mathfrak{g} = \mathfrak{n} \oplus \mathfrak{t} \oplus \mathfrak{n}^+$ yields the $\emph{Harish-Chandra}$ $\emph{homomorphism}$:

\begin{gather*}
\phi: U(\mathfrak{g})^{\textbf{G}} \to U(\mathfrak{t})
\end{gather*}

This is a morphism of deformable $R$-algebras, so by applying the deformation and $\pi$-adic completion functors, one obtains the $\emph{deformed Harish}$-$\emph{Chandra homomorphism}$:

\begin{gather*}
\widehat{\phi_{n,K}}: \widehat{U(\mathfrak{g})^\textbf{G}_{n,K}} \to \widehat{U(\mathfrak{t})_{n,K}}
\end{gather*}

which we will denote via the shorthand $\widehat{\phi}: Z \to \widetilde{Z}$. We have an action of the Weyl group $\textbf{W}$ on the dual Cartan subalgebra $\mathfrak{t}^*_K$ via the shifted dot-action:

\begin{gather*}
w \bullet \lambda = w(\lambda + \rho^{\prime}) - \rho^{\prime}
\end{gather*}

where $\rho^{\prime}$ is equal the half-sum of the T-roots on  $\mathfrak{n}^+$. Viewing $U(\mathfrak{t})_K$ as an algebra of polynomial functions on $\mathfrak{t}^*_K$, we obtain a dot-action of $\textbf{W}$ on $U(\mathfrak{t})_K$. This action preserves the $R$-subalgebra $U(\mathfrak{t})_n$ of $U(\mathfrak{t})_K$ and so extends naturally to an action of $\textbf{W}$ on $\widetilde{Z}$. \\

\begin{thm}\label{9.3}
Suppose that that $\textbf{G} = PGL_n$, $p|n$, and $n > 2$. Then: \\

(a) set $A := \widehat{U(\mathfrak{g})_{n,K}}$. The algebra $Z$ is contained in the centre of $A$. \\

(b) the map $\widehat{\phi}$ is injective, and its image is the ring of invariants $\widetilde{Z}^{\textbf{W}}$. \\

(c) the algebra $\widetilde{Z}$ is free of rank $|\textbf{W}|$ as a module over $\widetilde{Z}^{\textbf{W}}$. \\

(d) $\widetilde{Z}^{\textbf{W}}$ is isomorphic to a Tate algebra $K \langle S_1, \cdots, S_l \rangle$ as complete doubly filtered $K$-algebras.
\end{thm}

\begin{proof}
(a): The algebra $U(\mathfrak{g})^{\textbf{G}}_K$ is central in $U(\mathfrak{g})_K$ via \cite[Lemma 23.2]{Hu3}. Since $U(\mathfrak{g})_K$ is dense in $A$, it is also contained in the centre of $A$. But $U(\mathfrak{g})^{\textbf{G}}_K$ is also dense in $Z$, and so $Z$ is central in $A$. \\

(b): By the Harish-Chandra homomorphism (see \cite[Theorem 7.4.5]{Di}), $\phi$ sends $U(\mathfrak{g})^{\textbf{G}}_K$ onto $U(\mathfrak{t})^{\textbf{W}}_K$, and so $\widehat{\phi}(Z)$ is contained in $\widetilde{Z}^{\textbf{W}}$. This is a complete doubly filtered algebra whose associated graded ring $\text{Gr}(\widetilde{Z}^{\textbf{W}})$ can be identified with $S(\mathfrak{t}_k)^{\textbf{W}_k}$. This induces a morphism of complete doubly filtered $K$-algebras $\alpha: Z \to \widetilde{Z}^{\textbf{W}}$. Its associated graded map $\text{Gr}(\alpha): \text{Gr}(Z) \to \text{Gr}(\widetilde{Z}^{\textbf{W}})$ can be identified with the isomorphism $\psi_k: S(\mathfrak{g}_k)^{\textbf{G}_k} \to S(\mathfrak{t}_k)^{\textbf{W}_k}$ by Proposition \ref{6.9}. Hence $\text{Gr}(\alpha)$ is an isomorphism, and so $\alpha$ is an isomorphism by completeness. \\

(c): By Theorem \ref{polynomial theorem} and Theorem \ref{G2 polynomial}, $S(\mathfrak{t}_k)$ is a free graded $S(\mathfrak{t}_k)^{\textbf{W}_k}$-module of rank $|\textbf{W}|$. Hence, by \cite[Lemma 3.2(a)]{AW}, $\widetilde{Z}$ is finitely generated over $Z$, and in fact is free of rank $|\textbf{W}|$. \\

(d): By Theorem \ref{polynomial theorem} and Theorem \ref{G2 polynomial}, $S(\mathfrak{t}_k)^{\textbf{W}_k}$ is a polynomial algebra in $l$ variables. Fix double lifts $s_1, \cdots, s_l \in U(\mathfrak{t})^{\textbf{W}}$ of these generators, as in the proof of Proposition \ref{6.9}. Define an $R$-algebra homomorphism $R[S_1, \cdots, S_l] \to \widetilde{Z}^{\textbf{W}}$ which sends $S_i$ to $s_i$. This extends to an isomorphism $K \langle S_1, \cdots, S_l \rangle \to \widetilde{Z}^{\textbf{W}}$ of complete doubly filtered $K$-algebras.
\end{proof} 

We identify the $k$-points of the scheme $\mathfrak{g}^* := \text{Spec}(\text{Sym}_R \mathfrak{g})$ with the dual of the $k$-vector space $\mathfrak{g}$, so $\mathfrak{g}^*(k) = \mathfrak{g}^*_k$. Let $G$ denote the $k$-points of the algebraic group scheme $\textbf{G}$. $G$ acts on $\mathfrak{g}_k$ and $\mathfrak{g}^*_k$ via the adjoint and coadjoint action respectively.

Recall the definition of the enhanced moment map $\beta: \widetilde{T^*\mathcal{B}}(k) \to \mathfrak{g}^*_k$ from Definition \ref{enhanced moment map}. Given $y \in \mathfrak{g}^*_k$, write $G.y$ to denote the $G$-orbit of $y$ under the coadjoint action. We write $\mathcal{N}$ (resp. $\mathcal{N}^*$) to denote the nilpotent cone (resp. dual nilpotent cone) of the $k$-vector spaces $\mathfrak{g}_k$ and $\mathfrak{g}^*_k$. \\

\begin{prop}\label{9.8}
Suppose $p$ is nonspecial for $G$. For any $y \in \mathcal{N}^*$, we have $\text{dim } \beta^{-1}(y) = \text{dim } \mathcal{B} - \frac{1}{2} \text{dim } G.y$.
\end{prop}

\begin{proof}
This is stated for $\mathcal{N}$ as \cite[Theorem 10.11]{JN}. The result follows by applying the $G$-equivariant bijection $\kappa: \mathcal{N} \to \mathcal{N}^*$ from \cite[Section 6.5]{JN}.
\end{proof}

We now let $\mathfrak{g}_{\mathbb{C}}$ denote the complex semisimple Lie algebra with the same root system as $G$, and let $G_{\mathbb{C}}$ be the corresponding adjoint algebraic group. By \cite[Remark 4.3.4]{CM}, there is a unique non-zero nilpotent $G_{\mathbb{C}}$-orbit in $\mathfrak{g}^*_{\mathbb{C}}$, under the coadjoint action, of minimal dimension. Since each coadjoint $G_{\mathbb{C}}$-orbit is a symplectic manifold, it follows that each of these dimensions is an even integer. We set:

\begin{gather*}
r := \frac{1}{2} \text{min } \lbrace \text{dim } G_{\mathbb{C}} \cdot y \mid 0 \neq y \in \mathfrak{g}_{\mathbb{C}} \rbrace
\end{gather*}

\begin{prop}\label{9.9}
For any non-zero $y \in \mathcal{N}^*, \frac{1}{2} \text{dim } G \cdot y \geq r$, with no restrictions on $(G,p)$.
\end{prop}

\begin{proof}
We will demonstrate that this inequality holds for all split semisimple algebraic groups $G$ defined over an algebraically closed field $k$ of positive characteristic. When the characteristic $p$ is small, we will proceed via a case-by-case calculation of the maximal dimension of the centraliser $Z_G(y)$ of $y \in \mathcal{N}^*$. \\

By Proposition \ref{9.8}, $\text{dim } \beta^{-1}(y) = \text{dim } \mathcal{B} - \frac{1}{2} \text{dim } G \cdot y$. We may assume $y \in \mathcal{N}$ and $G$ acts on $\mathfrak{g}$ via the adjoint action by \cite[Section 6.5]{JN}. By \cite[Theorem 2]{HS}, we see that:

\begin{gather*}
\text{dim }\beta^{-1}(y) = \frac{1}{2}(\text{dim } Z_G(y) - \text{rk }(G))
\end{gather*}

where $Z_G(y)$ denotes the centraliser of $y$ in $G$. Hence it suffices to demonstrate that the following inequality:

\begin{gather*}
\text{dim } \mathcal{B} - \frac{1}{2}(\text{dim } Z_G(y) - \text{rk }(G)) \geq r
\end{gather*}

holds in all types. We evaluate on a case-by-case basis, aiming to find the maximal dimension of the centraliser. We first note that, using the work of \cite[1.6]{Sm1}, we have the following table:

\begin{center}
\begin{tabular}{ c|c|c }
 Type & dim $\mathcal{B}$ & $r$ \\
 \hline
 $A_n$ & $1/2n(n+1)$ & $n$ \\ 
 $B_n$ & $n^2$ & $2n-2$ \\  
 $C_n$ & $n^2$ & $n$ \\
 $D_n$ & $n^2 - n$ & $2n-3$ \\
 $E_6$ & 36 & 11 \\
 $E_7$ & 63 & 17 \\
 $E_8$ & 120 & 29 \\
 $F_4$ & 24 & 8 \\
 $G_2$ & 6 & 3   
\end{tabular}
\end{center}

By \cite[Theorem 2.33]{LS}, when $p$ is nonspecial, the dimension of the centraliser is independent of the isogeny type of $G$. \\

Since $p$ is always nonspecial for a group of type $A$, it therefore suffices to consider $Z_{\mathfrak{sl}_n}(y)$. Since $p$ is good and $SL_n$ is a simply connected algebraic group, by \cite[Lemma 2.15]{LS}, it suffices to consider the centraliser of a non-identity unipotent element in $SL_n$. Via the identification $GL_n(k) = SL_n(k)Z(GL_n(k))$, it is sufficient to compute $Z_{GL_n(k)}(u)$, for some unipotent matrix $u$. This dimension is bounded above by $n^2$, the dimension of $GL_n(k)$ as an algebraic group, and so we have the expression:

\begin{gather*}
\text{dim } \mathcal{B} - \frac{1}{2}(\text{dim } Z_G(y) - \text{rk }(G)) \\
\geq  \frac{1}{2} n(n+1) - \frac{1}{2}(n^2 - n)
\geq n.
\end{gather*}

Hence the inequality is verified in type $A$. \\

For the remaining classical groups, view $y \in \mathcal{N}$ as a nilpotent matrix, which without loss of generality may be taken to be in Jordan normal form. Let $m_1 \geq \cdots \geq m_r$ be the sizes of the Jordan blocks, with $\sum_{i=1}^r m_i = n$, the rank of the group. By \cite[Theorem 4.4]{He}, we have:

\begin{gather*}
\text{dim } Z_G(y) = \sum_{i=1}^r (im_i - \chi_V(m_i))
\end{gather*}

where $\chi_V$ is a function $\chi_V: \mathbb{N} \to \mathbb{N}$. It follows that:

\begin{gather*}
\text{dim } Z_G(y) \leq \sum_{i=1}^r im_i = \sum_{j=1}^n \sum_{i=j}^r m_i.
\end{gather*}

Since $m_1 \geq \cdots \geq m_r$ by construction, the maximum value of this sum is attained when $m_k = 1$ for all $k$. Hence we obtain the inequality $\text{dim } Z_G(y) \leq \frac{1}{2}n(n+1)$. Using this, it is easy to see that the required inequality holds except possibly in the cases $B_2, B_3, D_4$ and $D_5$. \\

For these cases, along with all exceptional cases, we directly verify that the inequality holds using the calculations on dimensions of centralisers in \cite[Chapter 8 and Chapter 22]{LS}. \\
\end{proof}

This allows us to prove our generalisation of \cite[Theorem 9.10]{AW}; a result on the minimal dimension of finitely generated modules over $\pi$-adically completed enveloping algebras. \\ 

\begin{defn}\label{canonical dimension def}
Let $A$ be a Noetherian ring. $A$ is $\emph{Auslander-Gorenstein}$ if the left and right self-injective dimension of $A$ is finite and every finitely generated left or right $A$-module $M$ satisfies, for $i \geq 0$ and every submodule $N$ of $\text{Ext}^i_A(M,A)$, $\text{Ext}^j_A(N,A) = 0$ for $j < i$. \\

In this case, the $\emph{grade}$ of $M$ is given by:

\begin{gather*}
j_A(M) := \text{inf} \lbrace j \mid \text{Ext}^j_A(M,A) \neq 0 \rbrace
\end{gather*}

and the $\emph{canonical dimension}$ of $M$ is given by:

\begin{gather*}
d_A(M) := \text{inj.dim}_A(A) - j_A(M).
\end{gather*}
\end{defn}

By the discussion in \cite[Section 9.1]{AW}, the ring $\widehat{U(\mathfrak{g})_{n,K}}$ is Auslander-Gorenstein and so it makes sense to define the canonical dimension function:

\begin{gather*}
d: \lbrace \text{finitely generated } \widehat{U(\mathfrak{g})_{n,K}}-\text{modules} \rbrace \to \mathbb{N}.
\end{gather*}

\begin{thm}\label{9.10}
Suppose $n > 0$ and let $M$ be a finitely generated $\widehat{U(\mathfrak{g})_{n,K}}$-module with $d(M) \geq 1$. Then $d(M) \geq r$.
\end{thm}

\begin{proof}
By \cite[Proposition 9.4]{AW}, we may assume that $M$ is $Z$-locally finite. We may also assume that $M$ is a $\widehat{\mathcal{U}^{\lambda}_{n,K}}$-module for some $\lambda \in \mathfrak{h}^*_K$, by passing to a finite field extension if necessary and applying \cite[Theorem 9.5]{AW}. \\

By Proposition $\ref{9.3}$(b), $\lambda \circ (i \circ \widehat{\phi}) = (w \bullet \lambda) \circ (i \circ \widehat{\phi})$ for any $w \in \textbf{W}$. Hence we may assume $\lambda$ is $\rho$-dominant by \cite[Lemma 9.6]{AW}. Hence $\text{Gr }(M)$ is a $\text{Gr }(\widehat{\mathcal{U}^{\lambda}_{n.K}}) \cong S(\mathfrak{g}_k) \otimes_{S(\mathfrak{g}_k)^{\textbf{G}_k}} k$-module by Theorem $\ref{6.10}$. If $\mathcal{M} := \text{Loc}^{\lambda}(M)$ is the corresponding coherent $\widehat{\mathcal{D}^{\lambda}_{n,K}}$-module in the sense of Definition \ref{twisted localisation}, then $\beta(\text{Ch}(\mathcal{M})) = \text{Ch}(M)$ via \cite[Corollary 6.12]{AW}. \\

Let $X$ and $Y$ denote the $k$-points of the characteristic varieties $\text{Ch}(\mathcal{M})$ and $\text{Ch}(M)$ respectively. Now $\text{Gr}(M)$ is annihilated by $S^+(\mathfrak{g}_k)^{\textbf{G}_k}$, and so $Y \subseteq \mathcal{N}^*$. We see that the map $\beta: T^*\mathcal{B} \to \mathfrak{g}$ maps $X$ onto $Y$. \\

Let $f: X \to Y$ be the restriction of $\beta$ to $X$. By \cite[Corollary 9.1]{AW}, since $\text{dim } Y = d(M) \geq 1$ we can find a non-zero smooth point $y \in Y$. By surjectivity, we have a smooth point $x \in f^{-1}(y)$. The induced differential $df_x: T_{X,x} \to T_{Y,y}$ on Zariski tangent spaces yields the inequality:

\begin{gather*}
\text{dim } Y + \text{dim } f^{-1}(y) \geq \text{dim } T_{X,x}
\end{gather*}

By \cite[Theorem 7.5]{AW}, $\text{dim } T_{X,x} \geq \text{dim } \mathcal{B}$. Hence:

\begin{gather*}
d(M) = \text{dim } Y \geq \text{dim } \mathcal{B} - \text{dim } \beta^{-1}(y)
\end{gather*}

By Proposition $\ref{9.8}$ and Proposition $\ref{9.9}$, the RHS equals $r$.
\end{proof}

$\textbf{Proof of Theorem C:}$ This follows from Theorem \ref{9.10} and \cite[Section 10]{AW} in the split semisimple case. We may then apply the same argument as in \cite{AJ} to remove the split hypothesis on the Lie algebra.

\bibliographystyle{plain}
\bibliography{Dmodulesbib}

\end{document}